\documentclass[11pt,leqno]{article}

\usepackage{amsmath,amssymb,amsthm}

\usepackage[margin=3cm]{geometry} 

\usepackage{titlesec,hyperref}

\usepackage{color}

\usepackage{fancyhdr}
\titleformat{\subsection}{\it}{\thesubsection.\enspace}{1pt}{}

\newtheorem{theo}{Theorem}[section]
\newtheorem{lemm}[theo]{Lemma}
\newtheorem{defi}[theo]{Definition}

\numberwithin{equation}{section}

\allowdisplaybreaks 

\begin{document}
\title{The existence  of the  global entropy weak solutions  for a generalized Camassa-Holm equation
\hspace{-4mm}
}

\author{Chunxia $\mbox{Guan}^1$\footnote{E-mail: guanchunxia123@163.com}, \quad \quad Xi $\mbox{Tu}^2$\footnote{E-mail: tuxi@mail2.sysu.edu.cn} \quad and\quad
 Zhaoyang $\mbox{Yin}^{3}$\footnote{E-mail: mcsyzy@mail.sysu.edu.cn}\\
 $^1\mbox{Department}$ of Mathematics,
Guangdong University of Technology,\\ Guangzhou, 510520, China\\
 $^2\mbox{Department}$ of Mathematics,
Foshan University,\\ Foshan, 528000, China\\
 $^3\mbox{Department}$ of Mathematics, Sun Yat-sen University,\\
    Guangzhou, 510275, China}

\date{}
\maketitle
\hrule

\begin{abstract}
In this paper, we prove the existence of a
global entropy weak solution $u\in H^1(\mathbb{R})$ and $\partial_{x}u\in L^1(\mathbb{R})\cap BV(\mathbb{R})$ for the Cauchy problem of a generalized Camassa-Holm equation by the viscous approximation method.

\vspace*{5pt}
\vspace*{5pt}
\\\end{abstract}

\vspace*{10pt}

\tableofcontents

\section{Introduction}
In this paper we consider the Cauchy problem of the following generalized Camassa-Holm equation:
\begin{align}\label{E1}
\left\{
\begin{array}{ll}
u_t-u_{txx}=\partial_x(2+\partial_x)[(2-\partial_x)u]^2,~~~~  t>0,\\[1ex]
u(0,x)=u_{0}(x),
\end{array}
\right.
\end{align}
which can be rewritten as
\begin{align}\label{E2}
\left\{
\begin{array}{ll}
 \partial_{t}m-(4u-2\partial_{x}u)\partial_{x}m&=2m^2
 +(8\partial_xu-4u)m+2(u+\partial_xu)^2\\&=F(m,u),\\[1ex]
 m(0,x)=m_{0},
\end{array}
\right.
\end{align}
where $m=u-u_{xx}.$

Note that $G(x)=\frac{1}{2}e^{-|x|}$ and $G(x)\star f =(1-\partial _x^2)^{-1}f$ for all $f \in L^2$ and $G \star m=u$. Then we can rewrite Eq.(\ref{E2}) as follows:
\begin{align}\label{E3}
\left\{
\begin{array}{ll}
u_t-4uu_x=-u^2_{x}+G\star[\partial_x(2u_x^2+6 u^{2})+u^2_{x}],~~~~  t>0,\\[1ex]
u(0,x)=u_{0}(x).
\end{array}
\right.
\end{align}

The Eq.(\ref{E1}) was proposed recently by Novikov in \cite{n1}. He showed that the equation (1.1) is integrable by using as definition of integrability the existence of an infinite hierarchy of quasi-local higher symmetries \cite{n1} and it belongs to the following class \cite{n1}:
\begin{align}\label{E002}
(1-\partial^2_x)u_t=F(u,u_x,u_{xx},u_{xxx}),
\end{align}
which has attracted much interest, particularly in the possible integrable members of (\ref{E002}).

The most celebrated integrable member of (\ref{E002}) is the well-known Camassa-Holm (CH) equation \cite{Camassa}:
\begin{align}
(1-\partial^2_x)u_t=3uu_x-2u_{x}u_{xx}-uu_{xxx}.
\end{align}
The CH equation can be regarded as a shallow water wave equation \cite{Camassa, Constantin.Lannes}.  It is completely integrable \cite{Camassa,Constantin-P,Constantin.mckean},
has a bi-Hamiltonian structure \cite{Constantin-E,Fokas}, and admits exact peaked solitons of the form $ce^{-|x-ct|}$ with $c>0$, which are orbitally stable \cite{Constantin.Strauss}. It is worth mentioning that the peaked solitons present the characteristic for the traveling water waves of greatest height and largest amplitude and arise as solutions to the free-boundary problem for incompressible Euler equations over a flat bed, cf. \cite{Camassa.Hyman,Constantin2,Constantin.Escher4,Constantin.Escher5,Toland}.

The local well-posedness for the Cauchy problem of the CH equation in Sobolev spaces and Besov spaces was discussed in \cite{Constantin.Escher,Constantin.Escher2,d1,Guillermo}. It was shown that there exist global strong solutions to the CH equation \cite{Constantin,Constantin.Escher,Constantin.Escher2} and finite time blow-up strong solutions to the CH equation \cite{Constantin,Constantin.Escher,Constantin.Escher2,Constantin.Escher3}. The existence and uniqueness of global weak solutions to the CH equation were proved in \cite{Constantin.Molinet, Xin.Z.P}. The global conservative and dissipative solutions of CH equation were investigated in \cite{Bressan.Constantin,Bressan.Constantin2}.

The second celebrated integrable member of (\ref{E002}) is the famous Degasperis-Procesi (DP) equation \cite{D-P}:
\begin{align}
(1-\partial^2_x)u_t=4uu_x-3u_{x}u_{xx}-uu_{xxx}.
\end{align}
The DP
equation can be regarded as a model for nonlinear shallow water
dynamics and its asymptotic accuracy is the same as for the
CH shallow water equation \cite{D-G-H}. The DP equation is integrable and has a bi-Hamiltonian structure \cite{D-H-H}. An inverse scattering approach for the
DP equation was presented in \cite{Constantin.lvanov.lenells,Lu-S}. Its
traveling wave solutions was investigated in \cite{Le}. \par The
local well-posedness of the Cauchy problem of the DP equation in Sobolev spaces and Besov spaces was established in
\cite{Coclite-Karlsen,G-L,H-H,y1}. Similar to the CH equation, the
DP equation has also global strong solutions
\cite{L-Y1,y2,y4} and finite time blow-up solutions
\cite{E-L-Y1, E-L-Y,L-Y1,L-Y2,y1,y2,y3,y4}. On the other hand, it has global weak
solutions \cite{C-K,E-L-Y1,y3,y4}.
\par
Although the DP equation is similar to the
CH equation in several aspects, these two equations are
truly different. One of the novel features of the DP
different from the CH equation is that it has not only
peakon solutions \cite{D-H-H} and periodic peakon solutions
\cite{y3}, but
also shock peakons \cite{Lu} and the periodic shock waves \cite{E-L-Y}.

The third celebrated integrable member of (\ref{E002}) is the known Novikov equation \cite{n1}:
\begin{align}
(1-\partial^2_x)u_t=3uu_{x}u_{xx}+u^2u_{xxx}-4u^2u_x.
\end{align}
The most difference between the Novikov equation and the CH and DP equations is that the former one has cubic nonlinearity and the latter ones have quadratic nonlinearity.

It was showed that the Novikov equation is integrable, possesses a bi-Hamiltonian structure, and admits exact peakon solutions $u(t,x)=\pm\sqrt{c}e^{|x-ct|}$ with $c>0$ \cite{Hone}.\\
$~~~~~~$ The local well-posedness for the Novikov equation in Sobolev spaces and Besov spaces was studied in \cite{Wu.Yin2,Wu.Yin3,Wei.Yan,Wei.Yan2}. The global existence of strong solutions under some sign conditions was established in \cite{Wu.Yin2} and the blow-up phenomena of the strong solutions were shown in \cite{Wei.Yan2}. The global weak solutions for the Novikov equation were studied in \cite{Wu.Yin}.

Recently, the Cauchy problem of Eq.(\ref{E1}) in the Besov spaces $B^{s}_{p,r},~s>max\{2+\frac{1}{p},~\frac{5}{2}\},~1\leq p,q \leq \infty$ and the critical Besov space $B^{\frac{5}{2}}_{2,1}$ has been studied in \cite{Tu-Yin1,Tu-Yin2}.
 The existence and uniqueness of global weak solutions under some certain sign
condition have been investigated in \cite{Tu-Yin3}.

Our aim of this paper is to prove the existence of a
global-in-time entropy weak solution $u\in H^1(\mathbb{R})$ and $\partial_{x}u\in BV(\mathbb{R})\cap L^{1}(\mathbb{R})$ to the Cauchy problem of Eq.(\ref{E1}) for any given initial data $u_0\in H^1(\mathbb{R})$ and $\partial_{x}u_0\in BV(\mathbb{R})\cap L^{1}(\mathbb{R})$. Before giving the
precise statement of our main result, we first introduce the
definition of a global weak solution to the Cauchy problem (1.1).

\begin{defi}\label{def1}(Global weak solution)
A function $u:\mathbb{R}\times\mathbb{R}\rightarrow \mathbb{R}$ is said to be an admissible global weak solution to the Cauchy problem (1.1) if
$$ u(t,x)\in L^{\infty}((0,\infty); H^{1}(\mathbb{R}))$$ satisfies Eq.(1.3) and $u(t,\cdot)\rightarrow
u_0$ as $t\rightarrow 0^+$ in the sense of distributions, that is, $\forall \phi\in \mathcal{C}_{c}^{\infty}([0,\infty)\times\mathbb{R}),$ there holds the equation
\begin{align}\label{e4}
\int_{\mathbb{R}_{+}}\int_{\mathbb{R}}(u\partial_{t}\phi-2u^2\partial_{x}\phi+\partial_{x}P_{1} \phi +\partial^2_{x}P_{2}\phi )dxdt+\int_{\mathbb{R}}u_{0}(x)\phi(0,x)dx=0,
\end{align}
where $P_{1}=G\ast[2(\partial_xu)^
2+6 (u^{2})],
~~P_{2}=G\ast[(\partial_xu)^2],
$
and for
any $t>0$,
$$\|u(t,\cdot)\|_{H^1(\mathbb{R})}\leq
\|u_0\|_{H^1(\mathbb{R})}.$$
\end{defi}

By extending the definition of a global weak solution by requiring some more (BV) regularity
and the fulfillment of an entropy condition we arrive at the notion of a global entropy weak
solution for the generalized Camassa-Holm equation.

\begin{defi} (Global entropy weak solution) We call a function $u : \mathbb{R}_+\times\mathbb{R} \rightarrow R$ a global entropy
weak solution of the Cauchy problem Eq.(1.1) provided

(i) $u$ is a global weak solution in the sense of Definition \ref{def1},

(ii) $\partial_{x}u \in L^\infty(0, T ; L^1 \cap BV(\mathbb{R}))$ for any $T >0$, and

(iii) for any convex $\mathcal{C}^2$ entropy $\eta: \mathbb{R} \rightarrow \mathbb{R}$ with corresponding entropy flux $q : \mathbb{R} \rightarrow \mathbb{R}$
defined by $q'(u)=u\eta'(u) $ there holds
$$\partial_{t}\eta(u) -4\partial_{x}q(u) -\eta'(u)
\partial_{x}P_{1,\varepsilon}  -\eta'(u)\partial^2_{x}P_{2,\varepsilon}\phi \leq0~in~
 \mathcal{D}'([0,\infty) \times R),$$
that is, $\forall\phi \in  C_c^\infty ([0,\infty) \times \mathbb{R})$, $\phi\geq0$,
\begin{align}\label{3-9}
\nonumber&\int_{\mathbb{R}_{+}}\int_{\mathbb{R}}(\eta(u)\partial_{t}\phi-4q(u)\partial_{x}\phi
+\eta'(u)\partial_{x}P_{1,\varepsilon} \phi
\\&+\eta'(u)\partial^2_{x}P_{2,\varepsilon}\phi) dx dt+\int_{\mathbb{R}}\eta'(u_{0}(x))\phi(0,x)dx\geq0.
\end{align}

\end{defi}

\par
The main result of this paper can be stated as follows:
\begin{theo} Let $u_{0}\in H^1(\mathbb{R})$ and $\partial_xu_0\in L^1(\mathbb{R})\cap BV(\mathbb R)$,  then
 Eq.(1.1) has a global entropy weak solution in the sense of
Definition 1.2.
 \end{theo}

The rest of the paper is organized as follows. In Section 2, we first construct the
approximate solutions sequence $u_\varepsilon$ as solutions to a viscous problem of
Eq.(1.1), and then derive the basic energy estimates on $u_\varepsilon$. In Section 3, we prove the strong convergence of
$\partial_xu_\varepsilon$ in $L^2_{loc}(\mathbb{R}_+\times\mathbb{R})$
and conclude the proof
of the main result.\\

\par
\noindent\textbf{Notations.} In the following, we denote by $\ast$
the spatial convolution. Given a Banach space $Z$, we denote its
norm by $\| \cdot\|_{Z}$. Since all spaces of functions are over
$\mathbb{R}$, for simplicity, we drop $\mathbb{R}$ in our notations
of function spaces if there is no ambiguity.

\vspace*{2em}
\section{Viscous approximate solutions}
In the section, we construct the approximate solution sequence
$u_\varepsilon=u_\varepsilon(t,x)$
as solutions to the viscous problem of Eq.(1.1), i.e.,
\begin{equation}\label{E01}
\left\{\begin{array}{ll} \partial_{t}u_{\varepsilon}
-4u_{\varepsilon}\partial_{x}u_{\varepsilon}
=\partial_{x}P_{1,\varepsilon} +\partial^2_{x}P_{2,\varepsilon} +\varepsilon u_{xx}, &t
> 0,\,x\in \mathbb{R},\\
 u(0,x) = u_{\varepsilon,0}(x),&x\in
\mathbb{R},
\end{array}\right.
\end{equation}
where $P_{1,\varepsilon}=G\ast[2(\partial_xu_{\varepsilon})^
2+6 (u_{\varepsilon}^{2})],
~~P_{2,\varepsilon}=G\ast[(\partial_xu_{\varepsilon})^2]
$
and $u_{\varepsilon,0}(x)=(\phi_\varepsilon\ast
u_0)(x)$, with $$
\phi_\varepsilon(x):=\left(\int_{\mathbb{R}}\phi(\xi)d\xi\right)^{-1}\frac{1}{\varepsilon}\phi(\frac{x}{\varepsilon}),\
 \ \ \ \ x\in\mathbb{R}, \ \ \varepsilon>0,$$
here $\phi\in C_c^{\infty}(\mathbb{R})$ is defined by
$$\phi(x)=\left \{\begin {array}{ll}e^{1/(x^2-1)},  \ \ \ \ \ \ \ \ |x|<1,
 \\ 0, \ \ \ \ \ \ \ \ \ \ \ \ \ \ \ \ \ \ |x|\geq 1.\end {array}\right.$$ If $u_0\in H^1, u'_0\in BV,$ then, for any $\varepsilon>0$, we have $u_{\varepsilon,0}\in
H^s(\mathbb{R}),  \forall s>\frac{5}{2}$,
$$\|u_{\varepsilon,0}\|_{H^1(\mathbb{R})}\leq \|u_0\|_{H^1(\mathbb{R})},~
\|u'_{\varepsilon,0}\|_{L^1(\mathbb{R})}\leq \|u'_0\|_{L^1(\mathbb{R})},~\|u'_{\varepsilon,0}\|_{BV(\mathbb{R})}\leq \|u'_0\|_{BV(\mathbb{R})}$$ and
$$ u_{\varepsilon,0}\rightarrow u_0, \ \text{in} \ H^1(\mathbb{R}).$$

\subsection{A priori estimates}

Before stating our theorem, let us first recall the following useful lemmas.

\begin{lemm}\label{Morse}
(Morse-type estimate, \cite{B.C.D,d1}) Let $s>\frac{d}{2}$. For any $a\in H^{s-1}(\mathbb{R}^d)$ and $b\in H^{s}(\mathbb{R}^d)$, there exists a constant $C$ such that
$$\|ab\|_{H^{s-1}(\mathbb{R}^d)}\leq C\|a\|_{H^{s-1}(\mathbb{R}^d)}\|b\|_{H^{s}(\mathbb{R}^d)}.$$
\end{lemm}

\begin{lemm}\label{3}\cite{B.C.D}
A constant C exists which satisfies the following properties.
If $s_1$ and $s_2$ are real numbers such that $s_1<s_2$ and $\theta \in (0, 1)$, then we
have, for any $(p,r)\in [1,\infty]^2$ and $u\in\mathcal{S}_{h}',$
\begin{align}
&\|u\|_{H^{\theta s_1+(1-\theta)s_2}}\leq\|u\|^{\theta}_{H^{s_1}}
\|u\|^{(1-\theta)}_{H^{s_2}}.
\end{align}
\end{lemm}
Now we introduce a priori estimates for the following transport equation.
 \begin{align}\label{20}
\left\{
\begin{array}{ll}
f_{t}+v\nabla f-\varepsilon f_{xx}=g,\\[1ex]
f|_{t=0}=f_{0},\\[1ex]
\end{array}
\right.
\end{align}
here $v,g$ and $f_0$ are given functions. Now, we recall an estimates of the solution to Eq.(2.3).

\begin{lemm}\label{est1}
(A priori estimates in Besov spaces, \cite{B.C.D,d1}) Let $s\geq \frac{d}{2}$ and $f$ be the solution of Eq.(2.3).
There exists a constant $C$ which depends only on $d,~s,$ we have

\begin{align}\label{10}
\|f\|_{H^{s}(\mathbb{R}^d)}\leq \bigg(\|f_0\|_{H^s(\mathbb{R}^d)}+\int^t_0
\|g(t')\|_{H^s(\mathbb{R}^d)}dt'\bigg)\exp(CV_{p_1}(t)),
\end{align}
where  {\small$V_{p_1}(t)=\displaystyle\int^t_0\|\nabla v\|_{H^{s-1}}(\mathbb{R}^d)d\tau$}, and $C$ is a constant depending only on $s$.
\end{lemm}

For any $\varepsilon>0,$ assume $u_{\varepsilon}$ is the solution of Eq.(2.1) with the initial data $u_{\varepsilon,0}\in  H^s~(s>\frac{5}{2})$, then we give some useful estimates about $u_{\varepsilon}$.
\begin{lemm}\label{Thm02}($H^1$ estimate)
Let $u_{0}\in H^1$.
Then the following inequality holds for any $t\geq0$:
\begin{align}\label{2-1}
\|u_{\varepsilon}\|_{H^{1}}
+\varepsilon\|\partial_{x}u_{\varepsilon}\|_{L^{2}}
+\varepsilon\|\partial_{xx}u_{\varepsilon}\|_{L^{2}
(\mathbb{R})}
\leq \|u_{0}(x)\|^2_{H^1}.
\end{align}
\end{lemm}
\begin{proof}
Define $q_{\varepsilon}=\frac{\partial u_{\varepsilon}}{\partial x}(t,x)$.
Differentiating  Eq. (\ref{E01}) with respect to $x$, we infer that

\begin{equation}\label{E02}
\left\{\begin{array}{ll}
\partial_{t} q_{\varepsilon}
+(2q_{\varepsilon}-4u_{\varepsilon})\partial_{x} q_{\varepsilon}-4q^2_{\varepsilon}- \partial^2_xP_{1\varepsilon}-\partial_xP_{2\varepsilon}=\varepsilon
\partial^2_{xx}q_{\varepsilon},&t> 0,\,x\in \mathbb{R},\\
 q_{\varepsilon}(0,x) = \frac{\partial u_{\varepsilon,0}}{\partial x}(x),&x\in\mathbb{R}.
\end{array}\right.
\end{equation}

Multiplying Eq.(\ref{E01})  by  $u_{\varepsilon}$ and  Eq.(\ref{E02})  by  $q_{\varepsilon}$ respectively, and then adding the resulting equations, we get
\begin{align}\label{2-2}
\nonumber \frac{1}{2}\frac{\partial}{\partial t}[u^2_{\varepsilon}+ q^2_{\varepsilon}]+&\frac{\partial}{\partial x}[\frac{2}{3}(u_{\varepsilon}^3+q_{\varepsilon}^3)-2u_{\varepsilon}q_{\varepsilon}^2
-u_{\varepsilon}(P_{1\varepsilon}+\partial_xP_{2\varepsilon})]=\\&
\varepsilon[\frac{\partial}{\partial x}(u_{\varepsilon}\frac{\partial u_{\varepsilon}}{\partial x}+q_{\varepsilon}\frac{\partial q_{\varepsilon}}{\partial x})-(\frac{\partial u_{\varepsilon}}{\partial x})^2-(\frac{\partial q_{\varepsilon}}{\partial x})^2].
\end{align}
Integrating (\ref{2-2}) over $[0,t]$ yields (\ref{2-1}). This completes the proof.
\end{proof}

\begin{lemm}\label{Thm08}($L^1$ estimate)
Let $u_{0}\in  H^1$, $u'_{0}\in L^1.$
Then we have for any  $t\geq0$:
\begin{align}\label{2-3}
\|\partial_{x}u_{\varepsilon}(t,\cdot)\|_{L^{1}}
\leq\|u'_{0}\|_{L^{1}}+8\|u_{0}\|^2_{H^1}t.
\end{align}
\end{lemm}

\begin{proof}
Let $\eta \in C^2(\mathbb{R})$ and $q : \mathbb{R}\rightarrow \mathbb{R} $ be such that
$q'(u) = u \eta'(u)$. By multiplying
the first equation in (\ref{E02}) with $\eta'(q_{\varepsilon})$ and using the chain rule, we get
\begin{align}\label{2-8}
\partial_{t} &\eta(q_{\varepsilon})+2\partial_{x}
(q(q_{\varepsilon}))
-4\partial_{x}(u_{\varepsilon}\eta(q_{\varepsilon}))
=4q^2_{\varepsilon}\eta'(q_{\varepsilon})\nonumber\\&-4q_{\varepsilon}
\eta(q_{\varepsilon})
+\partial^2_xP_{1\varepsilon}\eta'(q_{\varepsilon})
+\partial_xP_{2\varepsilon}\eta'(q_{\varepsilon})
+\varepsilon[\partial^2_{xx}(q_{\varepsilon})
-\eta''(q_{\varepsilon})(\partial_{x}q_{\varepsilon})^2].
\end{align}

Choosing $\eta(u) = |u|$ (modulo an approximation argument) and then integrating the
resulting equation over $\mathbb{R}$ yield
\begin{align}\label{2-4}
\partial_{t}\|q_{\varepsilon}\|_{L^1}\leq \|\partial^2_xP_{1\varepsilon}\|_{L^1}+\|\partial_xP_{2\varepsilon}\|_{L^1}
\leq6\|u_{0}\|_{H^1}.
\end{align}
Integrating (\ref{2-4}) over $[0,t],$ we get (\ref{2-3}). This completes the proof.
\end{proof}

\begin{lemm}\label{Thm03}(BV estimate in space)
Let $u_{0}\in H^1$,
$u'_{0}\in BV.$
Then we have for any $t\geq0$
\begin{align}\label{2-5}
\|\partial^2_{x}u_{\varepsilon}(t,\cdot)\|_{L^{1}}
\leq\|u'_{0}\|_{BV}+18\|u_{0}\|^2_{H^1}t.
\end{align}
\end{lemm}

\begin{proof}
Set $\gamma_{\varepsilon}=\partial_{x}q_{\varepsilon}
=\partial^2_{xx}u_{\varepsilon}.$ Then $\gamma_{\varepsilon}$ satisfies the following equation:
\begin{align}
\partial_{t}\gamma_{\varepsilon}+(2q_{\varepsilon}-4u_{\varepsilon})
\partial_{x}\gamma_{\varepsilon}
=-2\gamma^2_{\varepsilon}+4q_{\varepsilon}\gamma_{\varepsilon}-12
u_{\varepsilon}q_{\varepsilon}
+\partial_xP_{1\varepsilon}+\partial^2_xP_{2\varepsilon}
+\varepsilon\partial^2_{xx}(\gamma_{\varepsilon}).
\end{align}
 If $\eta \in C^2(\mathbb{R})$. By multiplying
the first equation in (\ref{E02}) with $\eta'(\gamma_{\varepsilon})$ and using the chain rule, we get
\begin{align}
\partial_{t} &\eta(\gamma_{\varepsilon})+\partial_{x}
[(2q_{\varepsilon}-4u_{\varepsilon})\eta(\gamma_{\varepsilon})]
\nonumber\\&\nonumber=(2\gamma_{\varepsilon}-4q_{\varepsilon})\eta(\gamma_{\varepsilon})
-2\gamma^2_{\varepsilon}\eta'(\gamma_{\varepsilon})
+4q_{\varepsilon}\gamma_{\varepsilon}\eta'(\gamma_{\varepsilon})
-12u_{\varepsilon}q_{\varepsilon}\eta'(\gamma_{\varepsilon})
)\nonumber\\&\nonumber+\partial_xP_{1\varepsilon}\eta'(\gamma_{\varepsilon})
+\partial^2_xP_{2\varepsilon}\eta'(\gamma_{\varepsilon})
+\varepsilon[\partial^2_{xx}\eta(\gamma_{\varepsilon})
-\varepsilon\eta''(\gamma_{\varepsilon})(\partial_{x}\gamma_{\varepsilon})^2.
\end{align}

Choosing $\eta(u) = |u|$ (modulo an approximation argument) and then integrating the
resulting equation over $\mathbb{R}$ yield
\begin{align}\label{2-6}
\partial_{t}\|\gamma_{\varepsilon}\|_{L^1}\leq 12\|u_{\varepsilon}q_{\varepsilon}\|_{L^1}+\|\partial_xP_{1\varepsilon}\|_{L^1}+\|\partial^2_xP_{2\varepsilon}\|_{L^1}
\leq18\|u_{0}\|_{H^1}.
\end{align}
Integrating (\ref{2-6}) over $[0,t],$ we get (\ref{2-5}). This completes the proof.
\end{proof}

\begin{lemm}\label{Thm04}($L^{\infty}$ estimate )
Let $u_{0}\in H^1$, $u'_{0}\in BV.$
Then we have for any $t\geq0$:
\begin{align}\label{2-7}
\|\partial_{x}u_{\varepsilon}(t,\cdot)\|_{L^{\infty}}
\leq\|u'_{0}\|_{BV}+18\|u_{0}\|^2_{H^1}t.
\end{align}
\end{lemm}

\begin{proof}
Since
$$
|q_{\varepsilon}|\leq \int_{\mathbb{R}}|\partial_{x}q_{\varepsilon}(t,x)|dy
=\|q_{\varepsilon}\|_{BV},
$$
(\ref{2-7}) is a direct consequence of (\ref{2-5}).
\end{proof}

\begin{lemm}\label{Thm05} (BV estimate in time)
Let $u_{0}\in H^1$, $u'_{0}\in BV.$
Then we have for any $t>0$:
\begin{align}\label{2-15}
\|\partial_{t}q_{\varepsilon}(t,\cdot)\|_{L^{1}}
\leq C_{t},
\end{align}
where $C_{t}=\frac{1}{t}\|u'_{0}\|_{BV}+3(\|u'_{0}\|_{BV}
+18\|u_{0}\|^2_{H^1}t)^2+4\|u_{0}\|^2_{H^1},$ is independent of $\varepsilon$ but dependent on $t$.
\end{lemm}

\begin{proof}
From (\ref{E01}), we obtain
\begin{align}
q_{\varepsilon}=e^{\varepsilon t\triangle}u'_{0}+\int_{0}^{t}
e^{\varepsilon (t-\tau)\triangle}
[(4u_{\varepsilon}-2q_{\varepsilon})\partial_{x} q_{\varepsilon}+4q^2_{\varepsilon}+ \partial^2_xP_{1\varepsilon}+\partial_xP_{2\varepsilon}]d\tau.
\end{align}
Since
\begin{align}
\|\partial_{t}e^{\varepsilon t\triangle}u'_{0}\|_{L^1}
=\nonumber&\|\partial_{t}(\frac{1}{\sqrt{4\pi\varepsilon t}}e^{-\frac{|x|^2}{4t\varepsilon}})\ast u'_{0}\|_{L^1}
\\\nonumber \leq&\|\frac{1}{2t\sqrt{4\pi\varepsilon t}}e^{-\frac{|x|^2}{4t\varepsilon}}\ast u'_{0}\|_{L^1}
+\|\frac{1}{t\sqrt{4\pi\varepsilon t}} \frac{|x|^2}{4t\varepsilon} e^{-\frac{|x|^2}{4t\varepsilon}}\ast u'_{0}\|_{L^1}
\\ \nonumber\leq&\|\frac{1}{2t\sqrt{4\pi\varepsilon t}} e^{-\frac{|x|^2}{4t\varepsilon}}\|_{L^1}
\|u'_{0}\|_{BV}+\|\frac{1}{t\sqrt{4\pi\varepsilon t}} \frac{|x|^2}{4t\varepsilon}e^{-\frac{|x|^2}{4t\varepsilon}}\|_{L^1}
\|u'_{0}\|_{BV}
\\\nonumber \leq&(\frac{1}{2\sqrt{\pi}t}\int_{\mathbb{R}}e^{-y^2}dy
+\frac{1}{\sqrt{\pi}t}\int_{\mathbb{R}}y^2e^{-y^2}dy)\|u'_{0}\|_{BV}
\\ \leq&\frac{1}{t}\|u'_{0}\|_{BV},
\end{align}
then it follows that
\begin{align}
\|\partial_{t}q_{\varepsilon}\|_{L^1}
\leq&\nonumber\|\partial_{t}e^{\varepsilon t\triangle}u'_{0}\|_{L^1}
+\|(4u_{\varepsilon}-2q_{\varepsilon})\partial_{x} q_{\varepsilon}+4q^2_{\varepsilon}+ \partial^2_xP_{1\varepsilon}+\partial_xP_{2\varepsilon}\|_{L^1}
\\\nonumber \leq&\|\partial_{t}e^{\varepsilon t\triangle}u'_{0}\|_{L^1}
+\|(4u_{\varepsilon}-2q_{\varepsilon})\|_{L^\infty}\|\partial_{x}q_{\varepsilon}\|_{L^1}
\\\nonumber &+4\|q_{\varepsilon}\|^2_{L^2}+\|\partial^2_xP_{1\varepsilon}\|_{L^1}
+\|\partial_xP_{2\varepsilon}\|_{L^1}
\\\nonumber \leq&\|\partial_{t}e^{\varepsilon t\triangle}u'_{0}\|_{L^1}
+(4\|u_{\varepsilon}\|_{H^1}+2\|q_{\varepsilon}\|_{L^\infty})
\|\partial_{x}q_{\varepsilon}\|_{L^1}+10\|u_{\varepsilon}\|^2_{H^1}
\\\nonumber \leq& \frac{1}{t}\|u'_{0}\|_{BV}+3(\|u'_{0}\|_{BV}
+18\|u_{0}\|^2_{H^1}t)^2+4\|u_{0}\|^2_{H^1}.
\end{align}
This completes the proof.
\end{proof}

\subsection{Global smooth approximate solutions}

We first establish global well-posedness of the Cauchy problem (\ref{E01}) in Sobolev spaces. Our main result can be stated as follows.
\begin{theo}
Assume $u_{\varepsilon,0}\in H^s,~s>\frac{5}{2},~~u'_{\varepsilon,0}\in L^1\cap BV$.
Then for any fixed $\varepsilon >0$,
 there exists a unique global smooth solution $u_{\varepsilon}\in
C(\mathbb{R}_+; H^s),~~s>\frac{5}{2}$,
to the Cauchy problem (\ref{E01}).
\end{theo}

The strategy of the proof of Theorem 2.9 is rather routine. We use several lemmas to prove this theorem. For the convenience of presentation, we will omit the subscript in $u_\varepsilon$ in the following proofs.

\begin{lemm}\label{local}
$u_{\varepsilon,0}\in H^s,~s>\frac{5}{2},~~u'_{\varepsilon,0}\in L^1\cap BV$. Then for any  fixed  $\varepsilon >0$ and $T>0$,
 there exists a unique smooth solution $u_{\varepsilon}\in
C((0,T); H^s),~~s>\frac{5}{2}$,
to the Cauchy problem (\ref{E01}).
\end{lemm}

\begin{proof}
In order to prove Lemma \ref{local}, we proceed as the following six steps.

\vspace*{4pt}\noindent {\bf Step 1.}
First, we construct approximate solutions which are smooth solutions of some linear equations. Starting for $m_0(t,x)\triangleq m(0,x)=m_0$, we define by induction sequences $(m_{n})_{n\in\mathbb{N}}$  by solving the following linear transport equations:
 \begin{align}\label{E00}
\left\{
\begin{array}{ll}
 &\partial_{t}m_{n+1}-(4u_{n}-2\partial_{x}u_{n})\partial_{x}m_{n+1}-\varepsilon \partial^2_{x}m_{n+1}\\&=2m_{n}^2
 +(8\partial_xu_n-4u_n)m_{n}+2(u_{n}+\partial_xu_{n})^2\\&=F(m_{n},u_{n}),\\[1ex]
 &m_{n+1}(t,x)|_{t=0}=S_{n+1}m_{0}.\\[1ex]
\end{array}
\right.
\end{align}

We assume that $m_n\in L^{\infty}(0,T;H^{s}),~~s>\frac{1}{2}$.

Since $s\geq\frac{1}{2}$, it follows that $H^{s}$ is an algebra, which leads to
$F(m_{n},u_{n})\break\in L^{\infty}(0,T;H^{s})$.
Hence, by the high regularity of $u$, (\ref{E00}) has a global solution $m_{n+1}$ which belongs to $C(0,T;H^{s})$ for all positive $T$.

\vspace*{4pt}\noindent{\bf Step 2.} Next, we are going to find some positive $T$ such that for this fixed $T$ the approximate solutions are uniformly bounded on $[0,T]$.
Since  Lemma \ref{Morse},  Lemma \ref{est1} (A priori estimates in Sobolev spaces), $H^{s}$($s>\frac{1}{2}$) is an algebra and $H^{s} \hookrightarrow L^{\infty}~(s>\frac{1}{2})$, we deduce that
\begin{align}\label{11}
\|m_{n+1}\|_{H^{s}}
\nonumber\leq& e^{C\int_{0}^{t}\|\partial_{x}(4u_n-2\partial_{x}u_n)\|_{H^{s}}d\tau}
\bigg(\|S_{n+1}m_0\|_{H^{s}}
+\int^t_0
\|F(m_{n},u_{n})\|_{H^{s}}dt'\bigg)
\\\nonumber \leq& e^{C\int^{t}_{0}\|m_n(\tau)\|_{H^{s}}d\tau}\bigg(\|S_{n+1}m_0\|_{H^{s}}
+\int^t_0\|F(m_{n},u_{n})\|_{H^{s}}dt'\bigg).
\\\leq& Ce^{C\|m_n(t)\|_{H^{s}}T}\bigg(\|m_0\|_{H^{s}}
+\|m_n\|^2_{H^{s}}T\bigg),
\end{align}
where we take $C\geq1.$

We fix a $T>0$ such that $e^{4C^2T\|m_0\|_{H^{s}}}\leq2,~~T\leq \frac{1}{16C^2\|m_0\|_{H^{s}}}.$ Suppose that
\begin{align}\label{013}
\|m_{n}(t)\|_{H^{s}}\leq 4C\|m_0\|_{H^{s}}\triangleq \mathbf{M},~~~~\forall t\in[0,T].
\end{align}

By using (\ref{11})-(2.10), we have
\begin{align}\label{8}
\quad\|m_{n+1}(t)\|_{H^{s}}\leq8 C\|m_0\|_{H^{s}}
=\mathbf{M}.
\end{align}
Thus, $(m_{n})_{n \in \mathbb{N}}$ is uniformly bounded in $L^{\infty}(0,T; H^{s})$.

\vspace*{4pt}\noindent{\bf Step 3.} From now on, we are going to prove that $m_n$ is a Cauchy sequence
 in $ L^{\infty}(0,T;H^{s-1})$. For this purpose, we deduce from (\ref{E00}) that
\begin{align}
\left\{
\begin{array}{ll}
&\partial_{t}(m_{n+l+1}-m_{n+1})-(4u_{n+l}-2\partial_{x}u_{n+l})\partial_{x}(m_{n+l+1}-m_{n+1})
 \\&=(u_{n+l}-u_{n})(\partial_{x}R_{n,l}^{1}+R_{n,l}^{2})
 +\partial_{x}(u_{n+l}-u_{n})
 (\partial_{x}R_{n,l}^{3}+R_{n,l}^{4})
 \\&~~+(m_{n+l}-m_{n})R_{n,l}^{5},\\[1ex]
 &m_{n+l+1}(t,x)-m_{n+1}(t,x)|_{t=0}=(S_{n+l+1}-S_{n+1})m_{0},\\[1ex]
\end{array}
\right.
\end{align}
where
\begin{align}
\nonumber&R_{n,l}^{1}=4m_{n+1}+2u_{n+l}+2u_{n},
\\\nonumber&R_{n,l}^{2}=-4m_{n}+2u_{n+l}+2u_{n},
\\\nonumber&R_{n,l}^{3}=-2m_{n+1}+2u_{n+l}+2u_{n},
\\\nonumber&R_{n,l}^{4}=8m_{n}+2u_{n+l}+2u_{n},
\\\nonumber&R_{n,l}^{5}=2m_{n+l}+2m_{n}+8\partial_{x}u_{n+l}-4u_{n+l}.
\end{align}

By Lemma \ref{Morse} (Morse-type estimate), Lemma \ref{est1} (A priori estimates in Sobolev spaces) and using the fact that $m_n$ is bounded in $L^{\infty}(0,T;H^{s})$, we infer that
\begin{align}\label{015}
&\quad\nonumber\|m_{n+m+1}(t)-m_{n+1}(t)\|_{H^{s-1}}\\ \nonumber\leq& C_{T}\bigg(\|(S_{n+m+1}-S_{n+1})m_0\|_{H^{s-1}}
\\\nonumber&+\int^t_0
\|(u_{n+l}-u_{n})\partial_{x}R_{n,l}^{1}\|_{H^{s-1}}
+\|(u_{n+l}-u_{n})R_{n,l}^{2}\|_{H^{s-1}}
\\\nonumber&+\|\partial_{x}(u_{n+l}-u_{n})\partial_{x}R_{n,l}^{3}\|_{H^{s-1}}
+\|\partial_{x}(u_{n+l}-u_{n})R_{n,l}^{4}\|_{H^{s-1}}
\\\nonumber&+\|(m_{n+l}-m_{n})R_{n,l}^{5}\|_{H^{s-1}}dt'\bigg)
\\\leq &C_{T}\bigg(\|(S_{n+l+1}-S_{n+1})m_0\|_{H^{s-1}}
+\int^t_0 C\mathbf{M}\|m_{n+l}-m_{n}\|_{H^{s-1}}dt'\bigg).
\end{align}

Since
$$\|\sum_{q=n+1}^{n+l}\Delta_{q}m_{0}\|_{H^{s-1}}\leq C2^{-n}\|m_{0}\|_{H^{s-1}} , $$ and that
$(m_n)_{n\in\mathbb{N}}$ is uniformly bounded in $L^{\infty}([0,T];H^{s})$, then it follows that
$$\|m_{n+l+1}(t)-m_{n+1}(t)\|_{H^{s-1}}\leq
C_T(2^{-n}+\int^{t}_{0}\|m_{n+l}-m_{n}\|_{H^{s-1}}d\tau).$$
It is easily checked by induction
$$\|m_{n+l+1}-m_{n+1}\|_{L^{\infty}(0,T; H^{s-1})}$$$$\leq \frac{(TC_T)^{n+1}}{(n+1)!}\|m_l-m_{0}\|_{L^{\infty}(0,T; H^{s-1})}
+C_T2^{-n}\sum_{k=1}^{n}2^{k}\frac{(TC_T)^{k}}{k!}.$$
Since $\|m_n\|_{L^{\infty}(0,T; H^{s-1})}$ is bounded independently of $n$, we can find a new constant $C'_T$ such that
$$\|m_{n+l+1}-m_{n+1}\|_{L^{\infty}(0,T; H^{s-1})}\leq C'_T2^{-n}.$$
Consequently, $(m_n)_{n\in\mathbb{N}}$ is a Cauchy sequence in $L^\infty(0,T;H^{s-1})$.
 Moreover it converges to some limit function $m\in L^\infty(0,T;H^{s-1}).$

\vspace*{4pt}\noindent{\bf Step 4.} We now prove the existence of solution. We prove that $m$ belongs to $C((0,T); H^s(\mathbb{R}))$ and satisfies Eq.(\ref{E01}) in the sense of distribution.
Since $(m_n)_{n\in\mathbb{N}}$ is uniformly bounded in $L^\infty(0,T;H^{s})$, the Fatou property for the Besov spaces guarantees that $m \in L^\infty(0,T;H^{s})$.

If $s'\leq s-1,$ then
\begin{align}\label{4}
\|m_n-m\|_{H^{s'}}\leq C\|m_n-m\|_{H^{s-1}}.
\end{align}
If $s-1\leq s'<s,$ by using Lemma \ref{3}, we have
\begin{align}\label{5}
\|m_n-m\|_{H^{s'}}\nonumber &\leq C\|m_n-m\|^\theta_{H^{s-1}}
\|m_n-m\|^{1-\theta}_{H^{s}}\\&
\leq C\|m_n-m\|^\theta_{H^{s-1}}(\|m_n\|_{H^{s}}+\|m\|_{H^{s}})^{1-\theta},
\end{align}
where $\theta=s-s'$.
Combining (\ref{4}) with (\ref{5}) for all $s'< s$, we have that $(m_n)_{n\in\mathbb{N}}$ converges to $m$ in $L^\infty([0,T];H^{s'})$. Taking limit in (\ref{E00}), we conclude that $m$ is indeed a solution of (\ref{E01}). Note that $m\in L^\infty(0,T;H^{s}).$
Then
\begin{align}\label{18}
\|2m^2+(8\partial_xu-4u)m+2(u+\partial_xu)^2\|_{H^{s}}
\leq C\|m\|^2_{H^{s}}.
\end{align}
This shows that $\partial_{x}P_{1,\varepsilon} +\partial^2_{x}P_{2,\varepsilon}$  also belongs to $L^\infty(0,T;H^{s}).$
Hence, 
 $m$ belongs to $C([0,T);H^s)$.


\vspace*{4pt}\noindent{\bf Step 5.} Finally, we prove the uniqueness and stability of solutions to Eq.(\ref{E01}). Suppose that $M=(1-\partial_x^2)u,~N=(1-\partial_x^2)v \in C((0,T); H^s(\mathbb{R}))$ are two solutions of (\ref{E01}).
Set $W=M-N.$ Hence, we obtain that
\begin{align}
\left\{
\begin{array}{ll}
&\partial_{t}W-(4u-2\partial_{x}u)\partial_{x}W
 \\&=(u-v)(\partial_{x}G^{1}+G^{2})
 +\partial_{x}(u-v)
 (\partial_{x}G^{3}+G^{4})+W G^{5},\\[1ex]
&W(t,x)|_{t=0}=M(0)-N(0)=W(0),\\[1ex]
\end{array}
\right.
\end{align}
where
\begin{align}
\nonumber&G^{1}=4N+2u+2v,
\\\nonumber&G^{2}=-4N+2u+2v,
\\\nonumber&G^{3}=-2N+2u+2v,
\\\nonumber&G^{4}=8N+2u+2v,
\\\nonumber&G^{5}=2M+2N+8\partial_{x}u-4u.
\end{align}

We define that
$U(t)\triangleq\int^{t}_{0}\|m(t')\|_{H^{s}}dt'$. By Lemma \ref{Morse} (Morse-type estimate), Lemma \ref{est1} (A priori estimates in Sobolev spaces) and using the fact that $m$ is bounded in $L^{\infty}(0,T;H^{s})$, we infer that
\begin{align}\label{15}
&\nonumber\|W(t)\|_{H^{s-1}}\\ \nonumber\leq& Ce^{CU(t)}\bigg{(}\|W(0)\|_{H^{s-1}}+\int^t_0
(\|(u-v)\partial_{x}G^{1}\|_{H^{s-1}}
+\|(u-v)G^{2}\|_{H^{s-1}}
\\\nonumber&+\|\partial_{x}(u-v)\partial_{x}G^{3}\|_{H^{s-1}}
+\|\partial_{x}(u-v)G^{4}\|_{H^{s-1}}+\|WG^{5}\|_{H^{s-1}})dt'\bigg{)}.
\\ \leq& C_{T} \mathbf{M}\|W\|_{H^{s-1}}.
\end{align}

Applying Gronwall's inequality yields
\begin{align}\label{19}
\sup_{t\in[0,T)}\|W(t)\|_{H^{s-1}}\leq e^{\widetilde{C_{T}}}\|W(0)\|_{H^{s-1}}.
\end{align}
In particular, $u_0=v_0$ in (\ref{19}) yields $u(t)=v(t).$ Consequently, we complete the proof.
\end{proof}

Now we present the blow-up scenario for the strong solutions to
Eq.(2.1).
\begin{lemm}\label{blow-up}
Let $u_{\varepsilon,0}\in H^{s}(\mathbb{R})$, $s>\frac{5}{2}$ be given and
assume that T is the maximal existence time of the corresponding
solution $u$ of Eq.(2.1) with the initial data $u_{\varepsilon,0}$. Then the
$H^s$-norm of $u(t,\cdot)$
blows up if and only if
$$\limsup_{t\rightarrow T}\{\|u_x(t,\cdot)\|_{L^{\infty}}
+u_{xx}(t,\cdot)\|_{L^{\infty}}\}=\infty.
$$
\end{lemm}

\begin{proof}
As the proof of Theorem 4.4 in \cite{Tu-Yin1}, we only need to consider
\begin{align}\label{22}
&\nonumber\frac{d}{dt}\int_{\mathbb{R}}[(\Lambda^{s-1}u)^2+(\Lambda^{s-1}u_x)^2]dx
\\\nonumber=&24\int_{\mathbb{R}}\Lambda^{s-1}u\Lambda^{s-1}(uu_x)dx
-8\int_{\mathbb{R}}\Lambda^{s}u\Lambda^{s}(uu_x)dx
\\\nonumber&-4\int_{\mathbb{R}}\Lambda^{s-1}u_{x}\Lambda^{s-1}(u_{x}u_{xx}))dx
-4\int_{\mathbb{R}}\Lambda^{s-1}u_{x}\Lambda^{s-1}(u^2_x)dx
\\&+2\varepsilon\int_{\mathbb{R}}\Lambda^{s-1}u\Lambda^{s-1}(u_{xx})dx
-2\varepsilon\int_{\mathbb{R}}\Lambda^{s-1}u\Lambda^{s-1}(u_{xxxx})dx.
\end{align}
Note that
\begin{align}
-8\int_{\mathbb{R}}\Lambda^{s}u\Lambda^{s}(uu_x)dx
\ \nonumber =&-8\int_{\mathbb{R}}\Lambda^{s}u[\Lambda^{s}(uu_x)-u_{x}\Lambda^{s}u]dx
-8\int_{\mathbb{R}}u_{x}(\Lambda^{s}u)^2dx
\\ \nonumber \leq& C \|u\|_{H^{s}}[\|u_{0}\|_{H^{1}}\|\Lambda^{s}u_{x}\|_{L^2}
+\|u_{xx}\|_{L^{\infty}}\|\Lambda^{s-1}u\|_{L^2(\mathbb{R})}]
+C\|u_x\|_{L^{\infty}}\|u\|^2_{H^{s}}
\\\nonumber\leq &C_{T,\varepsilon} \|u\|^2_{H^{s}}+ C_{T}(\|u_{x}\|_{L^{\infty}}+\|u_{xx}\|_{L^{\infty}})\|u\|^2_{H^{s}}
+2\varepsilon\|\Lambda^{s}u_{x}\|_{L^2},
\\ 2\varepsilon\int_{\mathbb{R}}\Lambda^{s-1}u\Lambda^{s-1}(u_{xx})dx
 \nonumber =&-2\varepsilon\int_{\mathbb{R}}\Lambda^{s-1}u_{x}\Lambda^{s-1}(u_{x})dx \leq C_{T,\varepsilon} \|u\|^2_{H^{s}},
\\ -2\varepsilon\int_{\mathbb{R}}\Lambda^{s-1}u\Lambda^{s-1}(u_{xxxx})dx
 \nonumber =&-2\varepsilon\int_{\mathbb{R}}\Lambda^{s-1}u_{xx}\Lambda^{s-1}(u_{xx})dx
= -2\varepsilon\|\Lambda^{s}u_{x}\|_{L^2}.
\end{align}
From the above inequalities and (2.30), we obtain
\begin{align}\label{23}
\frac{d}{dt}& \|u\|^2_{H^{s}}\leq
C_{T,\varepsilon} \|u\|^2_{H^{s}}(\|u_{xx}\|_{L^{\infty}}+\|u_{x}\|_{L^{\infty}}+1).
\end{align}

Applying Gronwall's inequality to (\ref{23}), we have
\begin{align}\label{24}
\|u\|^2_{H^{s}}\leq\|u_{\varepsilon,0}\|^2_{H^{s}}e^{C_{T,\varepsilon}\int_{0}^{T}(\|u_{xx}\|_{L^{\infty}}+\|u_{x}\|_{L^{\infty}}+1)d\tau}.
\end{align}
If $\|u_{x}\|_{L^{\infty}}+\|u_{xx}\|_{L^\infty}$ are bounded, from (\ref{24}), we know that $\|u(t,x)\|_{H^s}$ is bounded. By using the Sobolev embedding theorem $H^s(\mathbb{R})\hookrightarrow L^{\infty}(\mathbb{R}),~ s>\frac{1}{2},$ we have
\begin{align}\label{25}
\|u_{x}\|_{L^{\infty}}+\|u_{xx}\|_{L^\infty} \leq C\|u\|_{H^s},
\end{align}
with $s > \frac{5}{2}$. If $\|u\|_{H^s}$ is bounded, $s > \frac{5}{2},$ from (\ref{25}), we know that $\|u_{x}\|_{L^{\infty}}+\|u_{xx}\|_{L^\infty}$ are bounded. Moreover, if the maximal existence
time $T < \infty $ satisfies
$$\int^{T}_{0}\|u_{x}\|_{L^{\infty}}+\|u_{xx}\|_{L^{\infty}}d\tau<\infty,$$
we obtain from (\ref{24})
$$\lim_{t\rightarrow T}\sup_{0\leq \tau \leq t}\|u\|_{H^s}<\infty,$$
which contradicts with the assumption that $T < \infty$ is the maximal existence time.
 Thus we have
$$\int^{T}_{0}\|u_{x}\|_{L^{\infty}}+\|u_{xx}\|_{L^\infty}d\tau=\infty.$$

\end{proof}

By Lemma \ref{Thm04}  and Lemma  \ref{blow-up}, we only need to estimate on $\|u_{xx}\|_{L^\infty}$ which is equivalent to $\|w_x\|_{L^\infty}$, here $w=(2-\partial_x)u.$

\begin{lemm}\label{est2}
Let $T_1$ be the maximal existence time of the solution $u$ and $w=(2-\partial_x)u$.  For any $T<T_1$, we have
\begin{align}\label{31}
\|w_x(t,\cdot)\|_{L^\infty}\leq C(\varepsilon,T,\|w_{0}\|_{H^2}), \ for\ \forall t\in(0,T].
\end{align}
\end{lemm}

\begin{proof} Let $T_1$ be the maximal existence time of the solution $u$ and $w=(2-\partial_x)u$.  For any $T<T_1$,
by the Sobolev embedding theorem, we only have to prove the boundedness of $\|w_{x}(t,\cdot)\|_{H^1}$ for any $t\in (0,T].$
We have known that
\begin{align}
w_t+ww_x=-\partial_xG\ast(\frac{3}{2}w^2)+\varepsilon w_{xx}.
\end{align}
Then we have
\begin{align}
w_{tx}+w^2_x+ww_{xx}=\frac{3}{2}w^2-G\ast(\frac{3}{2}w^2)+\varepsilon w_{xxx},
\end{align}
and
\begin{align}
w_{txx}+3w_xw_{xx}+ww_{xxx}=3ww_x-\frac{3}{2}\partial_xG\ast(w^2)+\varepsilon w_{xxxx}.
\end{align}
Direct computations show that
\begin{align}
&\nonumber\frac{1}{2}\int_{\mathbb R}w^2_{x}dx+\varepsilon\int_{\mathbb R} w^2_{xx}dx\\
&=\nonumber-\int_{\mathbb R}(w_x^3+ww_xw_{xx})dx+\frac{3}{2}\int_{\mathbb R}w^2w_x-\int_{\mathbb R} w_xG\ast(\frac{3}{2}w^2)dx\\
&=\nonumber\int_{\mathbb R}ww_xw_{xx}dx+\frac{3}{2}\int_{\mathbb R}w^2w_x-\int_{\mathbb R} w_xG\ast(\frac{3}{2}w^2)dx\\
&\leq\frac{1}{2}\varepsilon\int_{\mathbb R} w^2_{xx}dx+C(\varepsilon,\|w\|_{L^\infty})\int_{\mathbb R}w^2_{x}dx+C(\|w\|_{L^\infty})\int_{\mathbb R}w^2_{x}dx+C(\|w\|_{L^\infty}),
\end{align}
and
\begin{align}
&\nonumber\frac{1}{2}\int_{\mathbb R}w^2_{xx}dx+\varepsilon\int_{\mathbb R} w^2_{xxx}dx\\
&=\nonumber-\int_{\mathbb R}(3w_xw_{xx}^2+ww_{xx}w_{xxx}-3ww_xw_{xx})dx-\frac{3}{2}\int_{\mathbb R} w_{xx}\partial_xG\ast(w^2)dx\\
&=\nonumber\int_{\mathbb R}(5ww_{xx}w_{xxx}+3ww_xw_{xx})dx-\frac{3}{2}\int_{\mathbb R} w_{xx}\partial_xG\ast(w^2)dx\\
&\leq\nonumber\frac{1}{2}\varepsilon\int_{\mathbb R} w^2_{xxx}dx+C(\varepsilon,\|w\|_{L^\infty})\int_{\mathbb R}w^2_{xx}dx\nonumber\\
&+C(\|w\|_{L^\infty})\int_{\mathbb R}(w^2_{x}+w^2_{xx})dx+C(\|w\|_{L^\infty}).
\end{align}
Adding the above two inequalities, by Gronwall's inequality and Lemma \ref{Thm02}, we obtain
\begin{align}
\|w_x(t,\cdot)\|_{H^1}\leq C(\varepsilon,T,\|w_0\|_{L^\infty}), \ for\ \forall t\in(0,T].
\end{align}
Using the Sobolev imbedding theorem, we can get (\ref{31}). This completes the proof of Lemma 2.12, which together with Lemma 2.10 implies Theorem 2.9.
\end{proof}

\section{Global entropy weak solutions}

In the section, making use of a priori estimates obtained in Section 2, we derive
the existence of entropy weak solutions to (\ref{E01}) under
 the assumption $u_{0} \in H^1,~ u'_{0} \in L^1 \cap BV$.

\subsection{Precompactness}

\begin{lemm}\label{com}
Under the assumption of Theorem 2.9, there exist a subsequence\\
$\{u_{\varepsilon_k}(t,x),P_{1\varepsilon_k}(t,x),P_{2\varepsilon_k}(t,x)\}$ of the sequence
$\{u_{\varepsilon}(t,x),P_{1\varepsilon}(t,x),P_{2\varepsilon}(t,x)\}$ and some functions
$u(t,x),P_{1}(t,x),P_{2}(t,x)$ with $u\in
L^\infty(\mathbb{R}_+,H^1)$,
$P_{1},P_{2}\in L^\infty(\mathbb{R}_+,W^{1,\infty})\cap L^\infty(\mathbb{R}_+,H^1),$ such
that
$$u_{\varepsilon_k}\rightarrow u,\   \ \ P_{1\varepsilon_k}\rightarrow
P_{1},\ and \ \ P_{2\varepsilon_k}\rightarrow P_{2}\ as\ k\rightarrow\infty, $$ uniformly on any compact subset
of $ \mathbb{R}_+\times\mathbb{R}$.
\end{lemm}

\begin{proof}
 It follows from Lemma 2.4 and Theorem 2.9 that $\{u_{\varepsilon}(t,x)\}$ is
uniformly bounded in $L^\infty(\mathbb{R}_+,H^1)$. Also, $\{\partial_tu_{\varepsilon}(t,x)\}$ is
uniformly bounded in $L^2([0,T]\times\mathbb{R})$ for $T>0.$ Indeed, by Lemma \ref{Thm02} and Eq.(\ref{E3}), we get
\begin{equation}\label{4-1-1}
\|P_{1\varepsilon}(t,\cdot)\|_{L^2}
\leq6\|G\|_{L^2}(\|u^2_{\varepsilon}(t,\cdot)\|_{L^1}+\|q_\varepsilon^2\|_{L^1})
\leq6\|u_{\varepsilon}(t,\cdot)\|^2_{H^1}\leq 6\|u_0\|^2_{H^1},
\end{equation}
\begin{equation}
\|\partial_xP_{1\varepsilon}(t,\cdot)\|_{L^2}
\leq6\|\partial_xG\|_{L^2}\|u_{\varepsilon}(t,\cdot)\|^2_{H^1}\leq 6\|u_0\|^2_{H^1},
\end{equation}
and
\begin{equation}\label{4-1-3}
\|P_{2\varepsilon}(t,\cdot)\|_{L^2}\leq\|G\|_{L^2}\|q_\varepsilon^2\|_{L^1}\leq\|u_{\varepsilon}(t,\cdot)\|^2_{H^1}\leq \|u_0\|^2_{H^1}.
\end{equation}
 Thus,
 there exist $u\in
C((0,T);L^\infty)$ and a
subsequence $\{u_{\varepsilon_k}(t,x)\}$ such that
$\{u_{\varepsilon_k}(t,x)\}$ is weakly compact in
$C((0,T);L^\infty)$ and
$\{u_{\varepsilon_k}(t,x)\}$ converges to $u(t,x)$ uniformly on each
compact subset of $\mathbb{R}_+\times\mathbb{R}$ as
$k\rightarrow\infty.$ Moreover,
 $u(t,x)\in C((0,T)\times\mathbb{R})\cap L^\infty((0,T);H^1).$

Next, we turn to the compactness of $\{P_{1\varepsilon}\}$. First,
by Lemma \ref{Thm02}, we have that $\{P_{1\varepsilon}\}$ is uniformly
bounded in $L^\infty(\mathbb{R}_+,H^1)$. Now we
estimate $\partial_tP_{1\varepsilon}.$ Note that
\begin{eqnarray*}
&&\frac{\partial P_{1\varepsilon}}{\partial
t}\\&=&G\ast\left(12u_\varepsilon\partial_tu_\varepsilon+4q_\varepsilon\partial_tq_\varepsilon\right)\\
&=&
12G\ast( u_\varepsilon\partial_tu_\varepsilon)
+4G\ast\left(q_\varepsilon\left(4u_\varepsilon\partial_xq_\varepsilon+2q_\varepsilon^2-2q_\varepsilon\partial_xq_\varepsilon+\varepsilon\partial_x^2q_\varepsilon-6u_\varepsilon^2
+P_{\varepsilon}+\partial_xG\ast q_\varepsilon^2\right)\right)\\
&=&I_1+I_2.
\end{eqnarray*}
By Lemma \ref{Thm02} and Eq.(\ref{E3}), we obtain that
$I_1=12G\ast(u_\varepsilon\partial_tu_\varepsilon)$
is uniformly bounded in $L^2([0,T]\times \mathbb{R})$ for any
$T>0$. Since $$4u_\varepsilon
q_\varepsilon\partial_xq_\varepsilon+2(q_\varepsilon)^3=2\left(u_\varepsilon
q_\varepsilon^2\right)_x,$$ and
$q_\varepsilon\partial_x^2q_\varepsilon=\partial_x(q_\varepsilon\partial_xq_\varepsilon)-(\partial_xq_\varepsilon)^2,$
it follows that
\begin{eqnarray*}
I_2&=&
4G\ast\left(q_\varepsilon\left(4u_\varepsilon\partial_xq_\varepsilon+2q_\varepsilon^2-2q_\varepsilon\partial_xq_\varepsilon+\varepsilon\partial_x^2q_\varepsilon-6u_\varepsilon^2
+P_{\varepsilon}+\partial_xG\ast q_\varepsilon^2\right)\right)\\&=&
4G\ast\left(q_\varepsilon(6u_\varepsilon^2+P_{\varepsilon}+\partial_xG\ast q^2_\varepsilon)-\varepsilon(\partial_xq_\varepsilon)^2\right)+4
G\ast\left(u_\varepsilon q_\varepsilon^2-\frac{2}{3}q_\varepsilon^3+\varepsilon
q_\varepsilon\partial_xq_\varepsilon\right)_x.
\end{eqnarray*}
By Lemmas \ref{Thm02} -- \ref{Thm05}, we get
\begin{eqnarray*}
&&\|I_2\|^2_{L^2((0,T)\times\mathbb{R})}\\&\leq& 4\int_0^T\|
G\|^2_{L^2}\|\left(q_\varepsilon(6u_\varepsilon^2+P_{\varepsilon}+\partial_xG\ast q^2_\varepsilon)-\varepsilon(\partial_xq_\varepsilon)^2\right)(t,\cdot)\|^2_{L^1}dt
\\&&+4\int_0^T\|\partial_xG\|^2_{L^2(\mathbb{R})}\|(u_\varepsilon q_\varepsilon^2-\frac{2}{3}q_\varepsilon^3+\varepsilon
q_\varepsilon\partial_xq_\varepsilon)(t,\cdot)\|^2_{L^1}dt\\&\leq&
C(T,\|u_0\|_{H^1},\|u'_0\|_{L^{1}}).
\end{eqnarray*}
Thus we prove that $\frac{\partial P_{1\varepsilon}}{\partial t}$ are
uniformly bounded in $ L^2((0,T)\times\mathbb{R})$ for every $T>0$.
Consequently,
 there exist $P_{1}\in
C((0,T);L^\infty)$ and a
subsequence $\{P_{1\varepsilon_k}(t,x)\}$ such that
$\{P_{1\varepsilon_k}(t,x)\}$ is weakly compact in
$C((0,T);L^\infty)$ and
$\{P_{1\varepsilon_k}(t,x)\}$ converges to $P_{1}(t,x)$ uniformly on
each compact subset of $\mathbb{R}_+\times\mathbb{R}$ as
$k\rightarrow\infty.$ Moreover,
 $P_{1}(t,x)\in C((0,T)\times\mathbb{R})\cap L^\infty((0,T);H^1).$

By H\"{o}lder's inequality, we have
\begin{eqnarray*}\label{4-1}
&&\|P_{1\varepsilon}(t,\cdot)\|_{L^\infty}\\&\leq&6\|G\|_{L^\infty}
\|(u_\varepsilon^2+q_\varepsilon^2)(t,\cdot)\|_{L^1}\leq6\|u_0\|^2_{H^1},
\end{eqnarray*}
and
\begin{eqnarray*}\label{4-2}
&&\|\partial_xP_{1\varepsilon}(t,\cdot)\|_{L^\infty}\\&\leq&6\|\partial_xG\|_{L^\infty}
\|(u_\varepsilon^2+q_\varepsilon^2)(t,\cdot)\|_{L^1}
\leq C\|u_0\|^2_{H^1}.
\end{eqnarray*}
Due to (\ref{4-1-1})-(\ref{4-1-3}) and the above estimates, we have $P_{1}\in
L^\infty(\mathbb{R}_+,W^{1,\infty})\cap L^\infty(\mathbb{R}_+,H^1).$

Since $P_{2\varepsilon}$ is a part of $P_{1\varepsilon}$, following the same proof of the compactness of $P_{1\varepsilon},$ we have that there exists $P_{2}\in
L^\infty(\mathbb{R}_+,W^{1,\infty})\cap L^\infty(\mathbb{R}_+,H^1),$  such that $P_{2\varepsilon_k}\rightarrow P_{2}\ as\ k\rightarrow\infty.$
This completes
the proof of the lemma.
\end{proof}

\subsection{Existence}

\begin{theo}
(Existence)
Assume that $u_{0}\in H^{1}$ and $u'_{0}\in L^{1}\cap BV,$ then there exists at least one entropy weak solution to (\ref{E01}).
\end{theo}
\begin{proof}
Assume that the approximating sequence
$\{u_{\varepsilon,0}\}_{\varepsilon>0}$ is chosen such that $u_{\varepsilon,0}\in H^s,~s>\frac{5}{2},~\|\partial_{x}u_{\varepsilon,0}\|_{L^1}
\leq\|u'_{0}\|_{L^1},~~\|\partial_{x}u_{\varepsilon,0}\|_{BV}
\leq\|u'_{0}\|_{BV}$. Then, due to Lemma \ref{Thm02} and Theorem 2.9,
we can verify that $\{u^{n}\}_{n\geq1}$ is uniformly bounded in the space $H^1((0; T) \times \mathbb{R})$.
Therefore we can get a sequence
such that as $k\rightarrow\infty$
\begin{align}
u^{\varepsilon_{k}}\rightharpoonup u ~~\text{weakly~~ in} ~~H^1((0; T) \times \mathbb{R})~~\text{for}~~ \varepsilon_{k}\rightarrow 0,
\end{align}
and
\begin{align}\label{l2}
u^{\varepsilon_{k}} \rightarrow ~u ~~a.e.~~ \text{on} ~~(0; T) \times \mathbb{R}~~\text{as}~~ \varepsilon_{k}\rightarrow 0,
\end{align}
for $u\in H^1((0; T) \times \mathbb{R})$.

And combining the priori estimates obtained in Section 2 together with the Helly theorem yields that
\begin{align}
\partial_{x}u_{\varepsilon_k}
\rightarrow \partial_{x}u  ~~a.e. ~~in~~ \mathbb{R}^{+} \times \mathbb{R}.
\end{align}
Hence
\begin{align}\label{32}
\partial_{x}u_{\varepsilon_k}
\rightarrow \partial_{x}u ~~ a.e.~~ in ~~ L^{p}((0,T)\times \mathbb{R}) ~~~~for~~~~\forall~~ T >0, ~~\forall~ ~ p \in [1,\infty).
\end{align}

To prove that the limit $u$ satisfies the
entropy inequality (\ref{3-9}), we need to know (\ref{32}) only for the case $p = 1$. Indeed, by (\ref{32}) (with $p = 1$) and Lemma \ref{com}, it follows by choosing $\varepsilon = \varepsilon_k$ in
\begin{align}\label{2-8}
\partial_{t} &\eta(u_{\varepsilon})-4\partial_{x}
(q(u_{\varepsilon}))
=\partial_xP_{1\varepsilon}\eta'(u_{\varepsilon})
+\partial^2_xP_{2\varepsilon}\eta'(u_{\varepsilon})
+\varepsilon[\partial^2_{xx}(u_{\varepsilon})
-\eta''(u_{\varepsilon})(\partial_{x}u_{\varepsilon})^2],
\end{align} in $D'([0,\infty) \times \mathbb{R}))$. Then sending $k\rightarrow\infty$, we complete the proof of Theorem 3.2.

\end{proof}

\noindent\textbf{Acknowledgements.} This work was
partially supported by NNSFC (No.11671407), Guangdong Special Support Program (No. 8-2015), and the key project of NSF of  Guangdong province (No. 2016A030311004).

\phantomsection
\addcontentsline{toc}{section}{\refname}

\end{document}